\documentclass{article}
\usepackage[utf8]{inputenc}  
\usepackage[T1]{fontenc}
\usepackage{amsmath,amssymb,graphicx,enumerate,latexsym,theorem}
\usepackage{todonotes}
\usepackage{tikz}
\usetikzlibrary{matrix}

\usepackage{blindtext}
\usepackage{wrapfig}

\newcommand{\F}{\mathcal F}
\newcommand{\A}{\mathfrak A}
\newcommand{\eps}{\varepsilon}

\newcommand{\R}{\mathbb R}
\newcommand{\B}{\mathcal B}
\renewcommand{\P}{\mathbb P}

\newcommand{\E}{\mathbb E}
\renewcommand{\d}{\,\mathrm{d}}

\DeclareMathSymbol{\mlq}{\mathord}{operators}{``}
\DeclareMathSymbol{\mrq}{\mathord}{operators}{`'}

\newcommand{\nin}{\not\in}
\newcommand{\nocomma}{}
\newcommand{\tmaffiliation}[1]{\\ #1}
\newcommand{\tmem}[1]{{\em #1\/}}

\newcommand{\tmstrong}[1]{\textbf{#1}}

\newenvironment{enumeratenumeric}{\begin{enumerate}[1.] }{\end{enumerate}}
\newenvironment{enumerateroman}{\begin{enumerate}[i.] }{\end{enumerate}}
\newenvironment{proof}{\noindent\textbf{Proof\ }}{\hspace*{\fill}$\Box$\medskip}
\newtheorem{definition}{Definition}
{\theorembodyfont{\rmfamily}}
\newtheorem{lemma}{Lemma}
{\theorembodyfont{\rmfamily}\newtheorem{remark}{Remark}}
\newtheorem{theorem}{Theorem}

\begin{document}

\title{Pointwise defined version of conditional expectation with respect to a random variable}

\author{
  Philipp Wacker
  \tmaffiliation{FAU Erlangen-Nürnberg}
}

\maketitle

\begin{abstract}
It is often of interest to condition on a singular event given by a random variable, e.g. $\{Y=y\}$ for a continuous random variable $Y$. Conditional measures with respect to this event are usually derived as a special case of the conditional expectation with respect to the random variables generating sigma algebra. The existence of the latter is usually proven via a non-constructive measure-theoretic argument which yields an only almost-everywhere defined quantity. In particular, the quantity $\E[f|Y]$ is initially only defined almost everywhere and conditioning on $Y=y$ corresponds to evaluating $\E[f|Y=y] = \E[f|Y]{Y=y}$, which is not meaningful because of $\E[f|Y]$ not being well-defined on such singular sets. This problem is not addressed by the introduction of regular conditional distributions, either. On the other hand it can be shown that the naively computed conditional density $f_{Z|Y=y}(z)$ (which is given by the ratio of joint and marginal densities) is a version of the conditional distribution, i.e. $\E[\{Z\in B\}|Y=y] = \int_B f_{Z|Y=y}(z) dz$ and this density can indeed be evaluated pointwise in $y$. This mismatch between mathematical theory (which generates an object which cannot produce what we need from it) and practical computation via the conditional density is an unfortunate fact. Furthermore, the classical approach does not allow a pointwise definition of conditional expectations of the form $\E[f|Y=y]$, only of conditional distributions $\E[\{Z\in B\}|Y=y]$. We propose a (as far as the author is aware) little known approach to obtaining a pointwise defined version of conditional expectation by use of the Lebesgue-Besicovich lemma without the need of additional topological arguments which are necessary in the usual derivation. 
\end{abstract}

\section{Introduction}
Conditioning with respect to singular events is a tricky subject. For concreteness, let's assume that we would like to condition the distribution of a random variable $Z$ on the singular event $Y=0$ for a different continuous random variable $Y$.

The usual path taken to give it well-defined meaning (see for example \cite{klenke2013probability} for a nice and classical exposition) is by defining \textit{conditional expectations} of a generic random variable $X$ with respect to a $\sigma$-algebra $\mathcal E$ first, in symbols $\E[X|\mathcal E]$. 

This object is fairly abstract and its existence is derived in a non-constructive way via the Radon--Nikodym theorem. Furthermore, it is only unique up to sets of measure $0$.

Then the notion of \textit{conditional probability} of $A$ with respect to $\mathcal E$ is obtained by setting $X = \chi_A$. This is then written as $\P(A|\mathcal E)$. Lastly, it is possible to further concretize to \textit{conditional distributions} by setting $A = \{Z\in B\}$, which is $\P(Z \in B| \mathcal E)$. There is some issue with the fact that the object $\P(Z\in B|\mathcal E)$ is defined only up to subsets of measure $0$ for a fixed set $B$. If we want to interpret this object as a function of $B$ we need to make sure that this ambiguity does not amplify to the full space (each set $B$ has an ambiguity set of measure $0$ but the (uncountable) union over all ambiguity sets of all sets $B$ of interest could amount to a non-negligible set). This problem is then solved by employing additional topological arguments, the notion of regular conditional distribution and the theory of Borel spaces (at which point most students, and many professional mathematicians as well, lose focus and decide to just believe the author that they know what they are doing). 

After having solved this problem, the fact remains that by derivation from its ``grandparent'' $\E[X|\mathcal E]$, the object $\P(Z\in B|\mathcal E)$ shares its unfavorable properties (non-constructive, non-uniqueness pointwise). If we now want to give meaning to $\P(Z \in B|Y=0)$, we need to concretize further by setting $\mathcal E = \sigma(Y)$. Then it can be shown that $\P(Z\in B|Y=y) = \P(Z \in B|\sigma(Y))(\{Y = y\})$ makes sense, but (due too its construction) \textit{not pointwise} in $y$. 

At this point the mathematical exposition is complete and the notion of conditional densities is revealed. This very easily computable quantity (which is just the ratio of the joint density of $Z,Y$ and the marginal density of $Y$) is then accepted as a ``version'' of the non-pointwise defined (and non-constructively introduced) conditional distribution. 

The fact that conditional distributions are not defined in a pointwise sense (i.e. on singular events) is highly unsatisfactory, because we are explicitly interested \textit{only} in pointwise evaluation on an event like $\{Y=0\}$.

The fundamental property of conditional expectation (see item ii. in definition \ref{def:condExp} can be shown to hold and from this point on the reader will rarely think about the original derivation of this object again and instead just use the computable conditional density instead. There is a certain mismatch here which some authors \cite{chang1997conditioning} have deplored. In particular it is unsatisfying that there is heavy mathematical machinery needed in order to non-constructively guarantee the existence of an object which is not defined pointwise (but which is only interesting for us in a pointwise sense), when in practice we only use the computable conditional density instead.

\begin{figure}[hbtp]
\centering
\includegraphics[width=\textwidth]{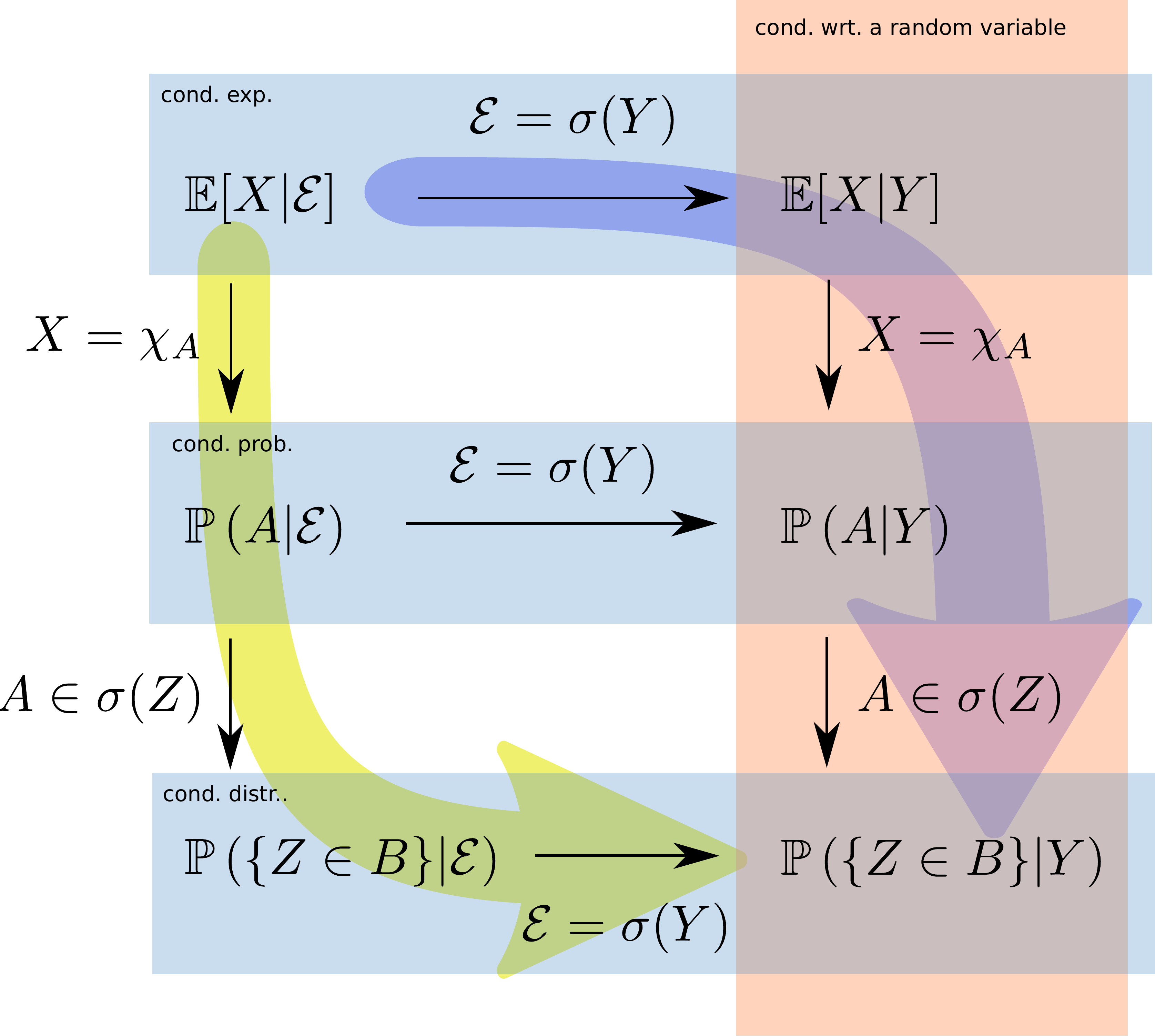}
\caption{Relationship of various conditionals. The notion of regular conditional distribution for example lives on the third row, whereas the main ideas of the manuscript live in the second column.}\label{fig:systematic}
\end{figure}

The underlying mathematical problem is the generality of the $\sigma$-algebra $\mathcal E$ we condition on. The classical derivation outlined so far follows the yellow (counterclockwise) arrow in figure \ref{fig:systematic}. This is the best we can do for general $\mathcal E$. But if we are explicitly interested in conditioning on a random variable $Y$, we can do much better, which is the point of this manuscript. By restricting to the case $\mathcal E = \sigma(Y)$ right away (i.e. following the blue, clockwise, arrow), we will show that it is possible to rigorously derive the existence of conditional distributions with respect to another random variable without the need for additional topological arguments. In particular, the notion of regular conditional distribution becomes unnecessary with this approach. Additionally, we can give conditional expectations, conditional probabilities and conditional distributions (all with respect to the random variable $Y$) a pointwise meaning in the sense that there is a canonical version of it which can be evaluated pointwise. The latter object is identical to the conditional densities used in practice. This procedure can be likened to the idea that while the constant function $f(x)=1$ as an object $f\in L^1([0,1])$ cannot be evaluated pointwise, it is possible to pick its continuous representative $f\in C([0,1])$ which can be evaluated pointwise.

This method is supposed to simplify an exposition of conditional distributions and can be used in a probability course. It also (to a large extent) resolves the mismatch between theoretical derivation and practical computation mentioned above. Of course, this higher specificity comes with the price of less generality as it requires us to pick a $\sigma$-algebra generated by a random variable but the author knows few interesting applications where this is not the case. If one needs conditional expectation in its full generality, then the classical exposition via regular conditional distributions etc. is of course still necessary.

The main part of the article starts with a recap of conditioning and then continues with the exposition of the proposed didactical approach. Readers familiar with the material can safely skim or skip sections \ref{sec:cond_1} and \ref{sec:cond_2}.

This manuscript's novelty does not originate in a lot of new mathematics but rather in the usefulness of its didactical approach. The mathematical methods are basic measure theory and some calculus which make them suitable for presentation in a standard probability course.

\section{Conditioning on regular events and with respect to
power set $\sigma$-Algebras} \label{sec:cond_1}

In order to keep the presentation self-contained and didactically continuous, we initially follow the lucid exposition in \cite{klenke2013probability}.

We assume that the reader is already familiar with basic conditional
probability theory with respect to non-singular events (e.g. ``Given the information that the sum of two dice is
$9$, what is the probability for the first dice to show a $5$?'') and that they
have an understanding (both mathematical and intuitive) of the following topics:

\begin{theorem}
  Law of total probability
  
  Let $(\Omega, \mathfrak{A}, \mathbb{P})$be a probability space and $(B_i)_{i
  \in I}$ an at-most countable collection of disjoint sets with $\mathbb{P}
  \left( \uplus_{i \in I} B_i \right) = 1$. Then for every event $A \in
  \mathfrak{A}$
  \[ \mathbb{P} (A) = \sum_{i \in I} \mathbb{P} (A|B_i) \cdot \mathbb{P} (B_i)
     . \]
\end{theorem}

\begin{theorem}
  Bayes' theorem
  
  Let $(\Omega, \mathfrak{A}, \mathbb{P})$be a probability space and $(B_i)_{i
  \in I}$ an at-most countable collection of disjoint sets with $\mathbb{P}
  \left( \uplus_{i \in I} B_i \right) = 1$. Then for every event $A \in
  \mathfrak{A}$ having probability $\mathbb{P} (A) > 0$ and every $k \in I$
  \[ \mathbb{P} (B_k |A) = \frac{\mathbb{P} (A|B_k) \cdot \mathbb{P}
     (B_k)}{\sum_{i \in I} \mathbb{P} (A|B_i) \cdot \mathbb{P} (B_i)} . \]
\end{theorem}

\begin{definition}\label{def:condExpRVRegular}
  Conditional expectation of random variables on regular events
  
  Let $X \in L^1 (\mathbb{P})$(i.e. X has a finite ``ordinary'' expectation)
  and $A \in \mathfrak{A}$ be an event with probability $\mathbb{P} (A) > 0$.
  Then we define
  \begin{equation}
    \label{eq1} \mathbb{E} [X|A] \equiv \int X (\omega) \mathbb{P} \left(
    \mathrm{d} \omega |A \right) = \frac{\mathbb{E} \left[ \mathrm{} 1_A
    \nocomma X \right]}{\mathbb{P} (A)} = \frac{\mathbb{E} [1_A X]}{\mathbb{E}
    [1_A]}
  \end{equation}
  For $A \in \mathfrak{A}$ with probability $\mathbb{P} (A) = 0$ we set
  $\mathbb{E} [X|A] = 0$.
\end{definition}

According to the last term we can interpret the conditional expectation as the
center of mass of $X$ on $A$, just like the common expectation $\mathbb{E}
[X]$ can be thought of as the center of ``probability mass'' on the whole
probability space. 
\begin{figure}[h!]
 \includegraphics[width=\textwidth]{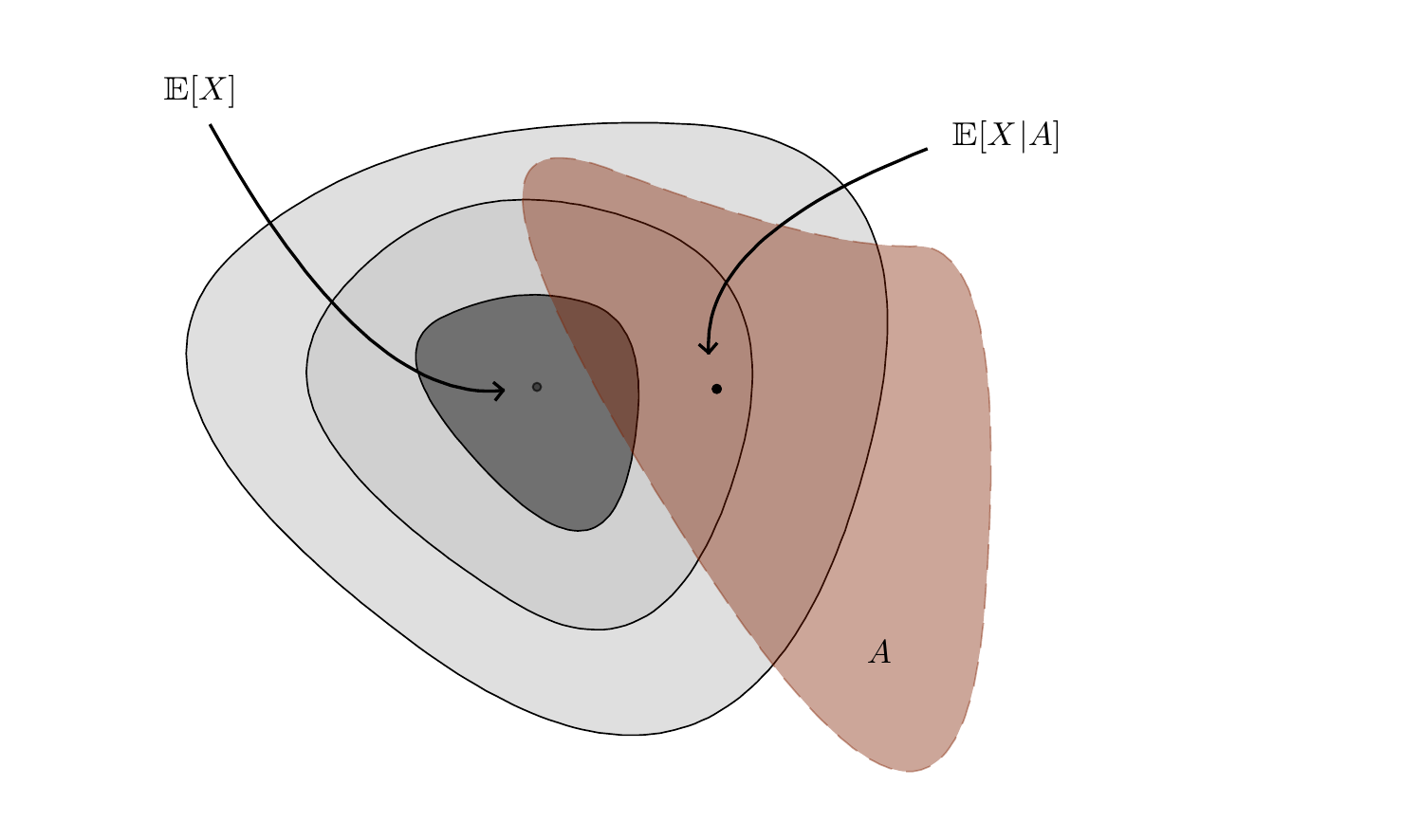}
  \caption{A visualization of conditional expectation: The contour plot in
  grey denotes contour lines of the density function. The expectation value
  will in this case be near the maximum of the density function as there is a
  lot of probability mass around it. The shade in red denotes a measurable set
  (event) $A$. The center of mass of $X$'s probability distribution
  conditioned on $A$ is depicted as well.}
\end{figure}

Having defined conditional expectations on specific events $A$ by $\mathbb{E}
[X|A]$ we could ask ourselves if we can generalize that notion to collections
of sets $A$, in particular $\sigma$-algebras.

Consider a common dice with six sides. We choose the probability space
canonically: $\Omega = \{ 1, 2, \ldots, 6 \}$ with elementary probabilities
$\mathbb{P} (\{ 1 \}) = \cdots =\mathbb{P} (\{ 6 \}) = \frac{1}{6}$. As
$\sigma$-algebras we take $\mathfrak{A}= \{ \emptyset, \{ 1, 2 \}, \{ 3,
\ldots, 6 \}, \Omega \}$, a $\sigma$-algebra ``unable to make distinctions''
for example between $1$ and $2$. Denote $A_1 = \{ 1, 2 \}$ and $A_2 = \{ 3, \ldots, 6 \}$ for
brevity:
\begin{eqnarray*}
  \mathbb{E} [X|A_1] & = & \frac{\mathbb{E} \left[ \mathrm{} 1_{A_1} \nocomma
  X \right]}{\mathbb{P} (A_1)} = \frac{\frac{1}{6} \cdot (1 + 2)}{\frac{2}{6}}
  = \frac{3}{2}\\
  \mathbb{E} [X|A_2] & = & \frac{\mathbb{E} \left[ \mathrm{} 1_{A_2} \nocomma
  X \right]}{\mathbb{P} (A_2)} = \frac{\frac{1}{6} \cdot (3 + 4 + 5 +
  6)}{\frac{4}{6}} = \frac{9}{2}
\end{eqnarray*}
The two remaining expectations are $\mathbb{E} [X| \emptyset] = 0$ and
$\mathbb{E} [X| \Omega] =\mathbb{E} [X] = \frac{7}{2}$.

\begin{figure}[h]
 \includegraphics[width=\textwidth]{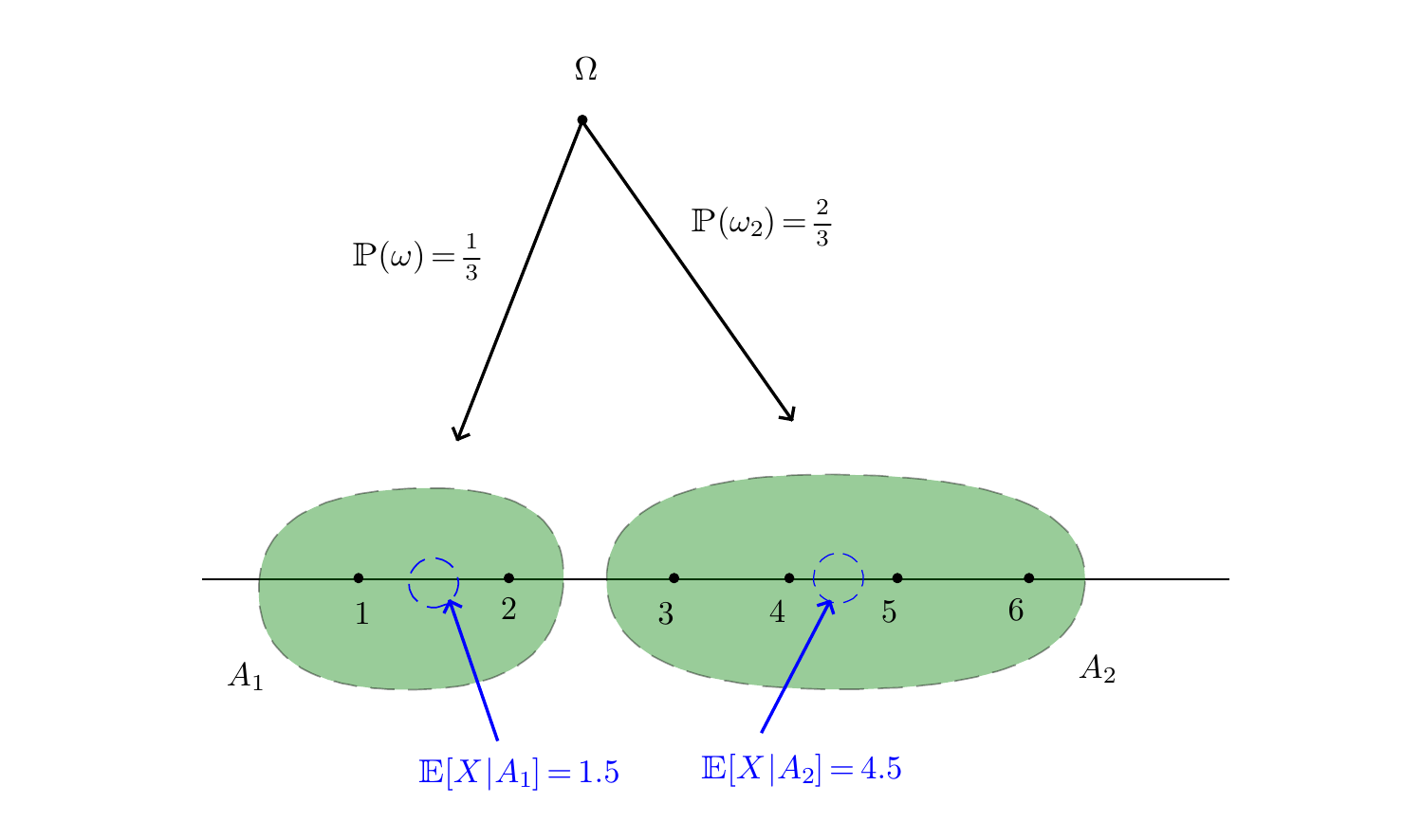}
  \caption{Conditional expectation with respect to a $\sigma$-Algebra as a
  random variable.}
\end{figure}

Formalized, this gives rise to the following definition:

\begin{definition}
  \label{def:condExpCountable}Conditional expectation with respect to a
  countable collection of events.
  Consider a probability space $(\Omega, \mathcal A, \P)$ and let $(B_i)_{i \in I}\subset \mathcal F$ be an {\tmstrong{at-most countable collection of
  disjoint sets with $\P(B_i)>0$ for all $i$}}  such that 
  \[\bigcup_{i \in I} B_i = \Omega.\]
  We define $\mathcal F := \sigma(\{B_i\})$, which is the power set of $\{B_i\}_i$.
  
  Let $X \in L^1 (\mathbb{P})$ be a random variable. We define {\tmem{the
  conditional expectation of}} $X$ {\tmem{given}} $\mathcal{F}$ as the
  {\tmstrong{random variable}}
  \[ \mathbb{E} [X|\mathcal{F}] (\omega) =\mathbb{E} [X|B_i] \quad
     \Leftrightarrow \quad \omega \in B_i . \]
\end{definition}

\begin{lemma}
  \label{lem:regCondExp}
  
  The random variable from Definition \ref{def:condExpCountable} has the
  following properties:
  \begin{itemize}
    \item $\mathbb{E} [X|\mathcal{F}]$ is measurable with respect to
    $\mathcal{F}$.
    
    \item $\mathbb{E} [X|\mathcal{F}] \in L^1 (\mathbb{P})$ and for every $A
    \in \mathcal{F}$
    \[ \int_A \mathbb{E} [X|\mathcal{F}] \mathrm{d} \mathbb{P}= \int_A X
       \mathrm{d} \mathbb{P}. \]
    In particular, $\mathbb{E} [\mathbb{E} [X|\mathcal{F}]] =\mathbb{E} [X]$
  \end{itemize}
\end{lemma}

\begin{remark}\label{rem:prog}
\begin{itemize}
\item Note that the measurability criterion is actually a \textit{restriction} on $\E[X|\mathcal F]$. In particular, $\E[X|\mathcal F]$ needs to be constant on the sets $B_i$.
\item Measurability and the integral condition are consequence of the definition of $\E[X|\mathcal F]$ here. For more general $\sigma$-algebras $\mathcal G$ which are not of the form as in definition \ref{def:condExpCountable}, these will be defining conditions rather than consequences.
\item The last property can be read as ``The mean value of all centers of masses of
  disjoint subsets is equal to the actual center of mass''.
  \item Note that we took the following sequence of steps while
  defining conditional probabilities:
  \begin{enumeratenumeric}
    \item Define conditional expectations $\mathbb{E} [X|A]$ on individual
    events $A \in \mathfrak{A}$.
    
    \item Generalize to conditional expectations $\mathbb{E} [X|\mathcal{F}]$
    on a (certain type of) $\sigma$-Algebra.
  \end{enumeratenumeric}
  This is a natural way of introducing conditional expectations on ``simple''
  events as the expectations $\mathbb{E} [X|A]$ are easily defined but the
  progression will be reversed for more general types of conditional
  expectations: The value of $\mathbb{E} [X|Y = y]$ for singular events $\{ Y
  = y \}$ will need to be derived from the notion of conditional expectation $\E[X|\sigma(Y)]$ which we still need to define.
  
\end{itemize}
  
\end{remark}

\section{Conditional expectation with respect to a general $\sigma$-Algebra} \label{sec:cond_2}

\subsection{Why do we need all that?}
Consider the following example: We are given a value $Y =y\in\R$ where the random variable $Y$ is modelled by $Y = X + \varepsilon$ with $X\sim N(0,\sigma^2)$ and $\varepsilon\sim N(0, \gamma^2)$. Intuitively, knowledge of the realization $Y=y$ should change our belief about $X$ and textbooks readily furnish the (very easy) formula that describes this updating process, yielding the posterior measure on $X|(Y=y)$. We can't solve that problem with our current machinery, though: The conditional cumulative distribution function is given by
\[ \mathbb{P} (X \leq x|Y = y) =\mathbb{E} [1_{\{ X \leq x \}} | Y=y ] \]
Tempted to use {\eqref{eq1}}, we would obtain an invalid expression: $\{ Y =
y \}$ is a {\tmem{singular event}} in our case, so its probability is $0$.
We will derive a better notion of conditional expectation for singular events
in the next section, but first we need conditional expectations on
$\sigma$-Algebras, as announced in Remark \ref{rem:prog}.

\subsection{Conditional expectations}

Let $(\Omega, \mathfrak{A}, \mathbb{P})$ be a probability space, $\mathcal{F}
\subset \mathfrak{A}$ be a $\sigma$-Algebra and $X \in L^1 (\Omega,
\mathfrak{A}, \mathbb{P})$. As announced before, measurability and integration property are definining conditions for candidates of a conditional expectation.
\begin{definition}
  \label{def:condExp}
  
  The random variable $Y$ is called {\tmem{conditional expectation of $X$
  given $\mathcal{F}$}}, in symbols $Z =\mathbb{E} [X|\mathcal{F}]$ if
  \begin{enumerateroman}
    \item $Z$ is measurable with respect to $\mathcal{F}$ and
    
    \item For every $A \in \mathcal{F}$ one has $\mathbb{E} [1_A X]
    =\mathbb{E} [1_A Z]$.
  \end{enumerateroman}
  For $B \in \mathcal{A}$ we call $\mathbb{P} [B|\mathcal{F}] \equiv
  \mathbb{E} [1_B |\mathcal{F}]$ the {\tmem{conditional probability of $B$
  given $\mathcal{F}$}}.
  
  For a random variable $Y$ we call $\mathbb{E} [X|Y] \equiv \mathbb{E} [X|
  \sigma (Y)]$ the conditional expectation of $X$ given $Y$.
\end{definition}

\begin{theorem}
  $\mathbb{E} [X|\mathcal{F}]$ exists and is unique a.s.
\end{theorem}

\begin{proof}
  For a proof see for example {\cite{klenke2013probability}}. This is always non-constructive via the Radon--Nikodym lemma.
\end{proof}

\begin{remark}
  We already defined the more elementary $\E[X|\mathcal G]$ for $\mathcal{G}= \sigma (\{ B_i \}_{i \in I})$
  with $\P(B_i)>0$ and  $\uplus_{i \in I} B_i = \Omega$. The new definition overloads this notation to a more general case of an arbitrary $\sigma$-algebra. This is justified because the elementary (and constructive) notion of conditional
  expectation from definition \ref{def:condExpCountable} was shown (in lemma \ref{lem:condexpptw}) to fulfil definition \ref{def:condExp} and is by uniqueness thus
  almost everywhere identical to the more general notion $\E[X|\mathcal G]$ from definition \ref{def:condExp}.
\end{remark}

\section{Conditioning with respect to singular events of a random variable $Y$}
In this (main) section we argue how conditional expectations with respect to a random variable can be defined pointwise and why this is not true for conditional expectations with respect to arbitrary $\sigma$-algebras in general.

Note that this has nothing to do with the notion of regular conditional probabilities, which is an idea located in the third row of figure \ref{fig:systematic}. We will in contrast talk about the second column of that figure.
\subsection{From $\E[X|Y]$ to $\E[X|Y=y]$}
We have established how to define $\E[X|Y] = \E[X|\sigma(Y)]$. We hope that this can help us to define a notion $\E[X|Y = y]$, i.e. a way of incorporating the \textit{information that the random variable $Y$ has actually attained the value $y$} into our knowledge about $X$.

The idea will be to take any $\omega \in \{\omega': Y(\omega') = y\}$ and set $\E[X|Y = y] = \E[X|Y](\omega)$. This needs to be justified, though. First we need the following

\begin{figure}
\begin{tikzpicture}
  \matrix (m) [matrix of math nodes,row sep=3em,column sep=4em,minimum width=2em]
  {
     (\Omega, \sigma(f)) & (\Omega', \mathcal A') \\
     (\mathbb{R}, \mathcal{B}(\mathbb{R})) &\\};
  \path[-stealth]
    (m-1-1) edge node [left] {$g$} (m-2-1)
            edge node [above] {$f$} (m-1-2)
    (m-1-2) edge node [below] {$\varphi$} (m-2-1);
\end{tikzpicture}  \hspace*{1cm}
\begin{tikzpicture}
  \matrix (m) [matrix of math nodes,row sep=3em,column sep=4em,minimum width=2em]
  {
     (\Omega, \sigma(Y)) & (\mathbb{R}, \mathcal{B}(\mathbb{R})) \\
     (\mathbb{R}, \mathcal{B}(\mathbb{R})) &\\};
  \path[-stealth]
    (m-1-1) edge node [left] {$\mathbb E[X|Y]$} (m-2-1)
            edge node [above] {$Y$} (m-1-2)
    (m-1-2) edge node [below] {$\varphi_X$} (m-2-1);
\end{tikzpicture}
\caption{Setting of the factorization lemma (left) and as used in our case (right).}
\end{figure}
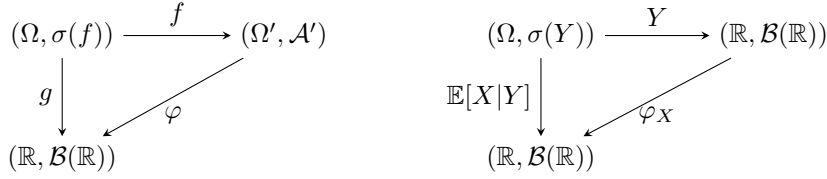

\begin{lemma}[Factorization lemma]   \label{lem:factorization}
  Let $(\Omega, \mathfrak{A})$ and $(\Omega', \mathfrak{A}')$ be two measure
  spaces and consider two maps $f : \Omega \rightarrow \Omega'$ and $g : \Omega \rightarrow \bar{\mathbb{R}} \equiv \mathbb{R} \cup \{
  \infty \}$. Then $g$ is $\sigma (f)$-measurable if and only if there is a measurable map $\varphi : (\Omega', \mathfrak{A}')
     \rightarrow (\bar{\mathbb{R}}, \mathcal{B} (\bar{\mathbb{R}}))$ such that
     \[ g = \varphi \circ f.\]
In this case we write symbolically $\varphi = g \circ f^{-1}$.
     \end{lemma}

\begin{figure}[h]
  \includegraphics[width=\textwidth]{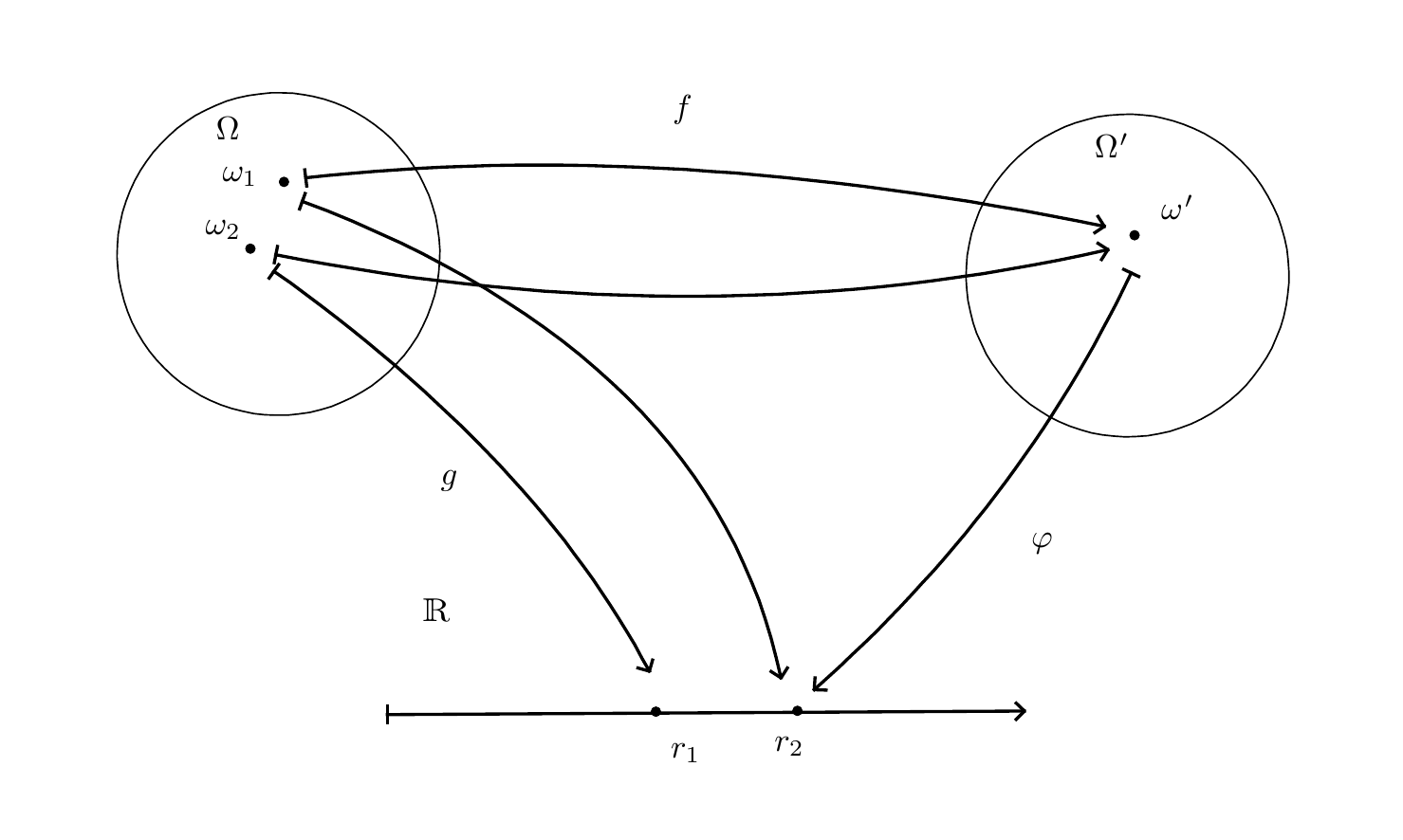}
  \caption{A setup not satisfying the assumptions of the factorization lemma: Assume $\Omega = \{
  \omega_1 \nocomma, \omega_2 \}$ and $\Omega' = \{ \omega' \}$. Choose
  standard $\sigma$-Algebras $\mathfrak{A}=\mathcal{P} (\Omega)$,
  $\mathfrak{A}' =\mathcal{P} (\Omega')$ and $\mathcal{B} (\mathbb{R})$. If $f$ and $g$ are defined as shown, there is no possible mapping $\varphi: \Omega' \to \R$ such that $g = \varphi \circ f$: The concatenation $\varphi \circ f$ is not equal to $g$, as $\varphi
  \circ f (\omega_2) = \varphi (\omega') = r_2$, whereas $g (\omega_2) = r_1$.
  This is due to the fact that $g$ is not $\sigma (f)$-measurable: $\sigma (f)
  = \{ f^{- 1} (A') |A' \in \mathfrak{A}' \} = \{ \emptyset, \{ \omega_1,
  \omega_2 \} \}$. Now for small $\varepsilon$, the set $R = (r_1 -
  \varepsilon, r_1 + \varepsilon)$ is open but $g^{- 1} (R) = \{ \omega_2 \}
  \nin \sigma (f)$. Intuitively, the problem is that $g$ and $f$ ``cluster''
  events in $\Omega$ differently: For $g$, both single events have different
  results whereas $f$ groups them together.}
\end{figure}

\begin{lemma}\label{lem:gcircfpreimage}
In the setting of the factorization lemma, $g(\{\omega\in \Omega: f(\omega) = y\})$ contains only one element. This means that the symbolical statement $\varphi = g\circ f^{-1}$  means $\varphi(\omega')$ is the unique element of the set $g(f^{-1}(\{\omega'\}))$.     
\end{lemma}
\begin{proof}
Assume that there are $r_1 \neq r_2$ such that $r_i \in g(\{\omega\in \Omega: f(\omega) = y\})$. Then there are distinct $\omega_1\neq \omega_2$ such that $f(\omega_1) = f(\omega_2)=y$ and $g(\omega_i) = r_i$, $i=1,2$. We consider the set $A = g^{-1}(\{r_1\})$. By construction we know that $\omega_1 \in A$ but $\omega_2 \not\in A$. Then $A\not\in \sigma(f)$ because $A$ cannot be written in the form $A = f^{-1}(B)$ for $B\in \B(\R)$ as $\omega_1\in A$ would immediately imply $\omega_2\in A$ due too $f(\omega_1) = f(\omega_2)$.
\end{proof}

Assume $Y : (\Omega, \mathfrak{A}, \mathbb{P}) \rightarrow (E, \mathcal{E})$
is a random variable with values in a measurable space $E$ and $Z =\mathbb{E}
[X| \sigma (Y)] : (\Omega, \mathfrak{A}, \mathbb{P}) \rightarrow \mathbb{R}$
be the conditional expectation of a random variable $X$. According to the
factorization lemma, there exists a map $\varphi_X : E \rightarrow \mathbb{R}$
such that $\varphi_X$ is $(\mathcal{E}, \mathcal{B} (\mathbb{R}))$-measurable
and $\varphi_X \circ Y =\mathbb{E} [X| \sigma (Y)]$. 

\begin{definition}
  \label{def:condExp2}Conditional expectation with respect to a continous
  random variable's results
  
  Let $X \in L^1 (\mathbb{P})$ and $Y : (\Omega, \mathfrak{A}) \rightarrow (E,
  \mathcal{E})$. Set $Z
  =\mathbb{E} [X|Y]$. Note that $Z$ is $\sigma(Y)$-measurable by definition, i.e. by lemma \ref{lem:factorization} there is a function $\varphi_X$ such that $\varphi_X \circ  Y = \E[X|Y]$. We call $\varphi_X$ the {\tmem{conditional expectation of $X$ given measurements of $Y$}}, in symbols $\mathbb{E} [X|Y = y] := \varphi_X(y)$. Similarly, we write $\mathbb{P}
  (A|X = x) =\mathbb{E} [1_A |X = x]$ if we choose $X = \chi_A$ for $A \in \mathfrak{A}$.  
  
 This definition does not look very constructive, but we can alternatively write (justified by lemma \ref{lem:gcircfpreimage}) $\E[X|Y=y] = \E[X|Y](\omega)$ for \textit{any} $\omega \in \{\omega\in \Omega: Y(\omega) = y\}$.
\end{definition}

  Although the definition of $\mathbb{E} [X|Y = y]$ seems to define exactly what we need, this pointwise evaluation $\varphi_X(y)$ is not yet meaningful at all: $\varphi_X$ is only defined almost everywhere. This is similar to the case of a function $F\in L^1([0,1])$. Even if $F(x) = 0$ (in the sense of an identity of $L^1$ functions), this does not mean we can tack down $F(0.5)$ or any other fixed point. The same situation presents itself here: We cannot with any meaning define $\E[X|Y=y]$ for any fixed value $y$ because we could change it to any other value with no harm to the object $\varphi_X = \E[X|Y = \cdot]$.
  
Rather, this definition of $\mathbb{E} [X|Y = y]$ has to be interpreted not for $y$ pointwise but similar to how we look at $L^1$ functions $f(``x``)$. This entry $x$ is to be understood as the dependent ``macro'' integration variable (used in integration contexts like property (ii) of definition \ref{def:condExp}), not as something we can look at with a microscope.
  
  On the other hand this is not consoling at all because we are \textit{only} interested in pointwise evaluations because of what we want to do, which is to use pointwise measurement information $\{Y = y\}$.
  
At this point, textbooks usually do one of two things:
\begin{enumerate}[a)]
\item Only look at the special case of conditional distributions $\P(X\leq x|Y=y)$, define the notion of regular conditional distribution and employ some slightly occult topological arguments to justify its existence.
\item Define conditional densities $f_{X|Y=y}(x) = \frac{f_{X,Y}(x,y)}{f_Y(y)}$ which make sense pointwise and show that they satisfy the criteria for a conditional probability. 
\end{enumerate}
Both approaches do \textit{not} solve the problem of conditional \textit{expectations} having no pointwise meaning.
  
In addition, there is a certain mismatch between rigorous derivation of conditional probabilities in a) and the ad-hoc way conditional densities (b)) are used in practiced, as has been remarked for example in \cite{chang1997conditioning}.  
  
  In order to facilitate this, we use the Lebesgue--Besicovich differentiation theorem. This will allows us to define a pointwise version of conditional expectation and we can show that this recovers the usual conditional density.
  \begin{lemma}[Lebesgue--Besicovich differentiation theorem]
  
  Let $f\in L^1(\R, \mu)$. Then for $\mu$-almost-all $x\in \R$, we have 
  \[ f(x) = \lim_{U_x\to x} \frac{1}{\mu(U_x)}\int_{U_x} f(y) \d\mu(y).\]
  where $U_x$ is an arbitrary sequence of neighborhoods with vanishing diameter around $x$.
  The function $\tilde f(x)$ defined by the right hand side (which is almost everywhere identical to $f$) is called the ``precise representative'' of $f$ (see \cite{evans2015measure}).
  \end{lemma}
  
  The preceding lemma tells us that we can pick the precise representative of $\varphi_X$ which we can indeed evaluate pointwise (by the approximation procedure given by the Lebesgue-Besicovich differentiation theorem). The beauty of this construction is that this ties together the object generated by the factorization lemma and the formula for conditional densities (which is usually proposed in an ad-hoc way and then proved to be compatible with the notion of conditional expectation by checking the assumptions). It shows in particular that there is a version of conditional expectation which is defined pointwise in a meaningful way.
  
  \begin{lemma} \label{lem:condexpptw}
  If $Y$ is such that the event $\{|Y-y|<\eps\}$ has positive measure for each $\eps > 0$, then we can pointwise define $\E[X|Y=y] := \lim_{\epsilon\to 0} \E[X|Y\in(y-\epsilon, y + \epsilon)]$. In particular $\P[A|Y = y] = \lim_{\epsilon\to 0} \P[A|Y\in(y-\epsilon, y + \epsilon)]$ and if the conditional distribution of a random variable $Z$ given $Y=y$ has density given by $f_{Z|Y=y}(z) = \frac{f_{Z,Y}(z,y)}{f_Y(y)}$ if $(Z,Y)$ has a continuous density.
  \end{lemma}
  \begin{proof} Consider the function $\varphi_X$ on the measure space $(\R, \B(\R), \mu\circ Y^{-1})$. Then by the Lebesgue--Besicovich differentiation theorem,
  \begin{align*}
 \varphi_X(y) &= \lim_{\epsilon\to 0} \frac1{(\mu\circ Y^{-1})((y-\epsilon, y + \epsilon))}\int_{y-\epsilon}^{y + \epsilon} \varphi_X(z)\d(\mu\circ Y^{-1})(z)\\
  &=  \lim_{\epsilon\to 0} \frac{\int_{\{Y \in(y-\epsilon, y + \epsilon)\} }\varphi_X(Y(\omega))\d\mu(\omega)}{\mu(\{Y\in (y-\epsilon, y+\epsilon)\}}\\
  &= \lim_{\epsilon\to 0} \frac{\int_{\{Y \in(y-\epsilon, y + \epsilon)\} }\E[X|Y](\omega)\d\mu(\omega)}{\mu(\{Y\in (y-\epsilon, y+\epsilon)\}}\\
  &= \lim_{\epsilon\to 0} \E[X|Y\in(y-\epsilon, y+\epsilon)]
  \end{align*}
  where the last step is due to definition \ref{def:condExpRVRegular} of conditional expectation with respect to the non-singular events $Y\in(y-\epsilon, y+\epsilon)$ which proves the first claim. The second claim is a direct result of setting $X = \chi_A$ and the third claim follows from further setting $A = \{Z \in B\}$. Then 
  \begin{align*}
  \P(\{Z\in B\}|Y=y) &= \lim_{\epsilon\to 0}\P(\{Z\in B\}|Y\in(y-\epsilon, y+\epsilon))\\
  &= \lim_{\epsilon\to 0} \frac{\int_{B\times (y-\epsilon, y+\epsilon)} f_{Z,Y}(z,\tilde y)d\tilde y dz}{\int_{(y-\epsilon, y+\epsilon)}f_Y(y) dy }\\
  &=  \frac{\int_B f_{Z,Y}(z, y)dz}{f_Y(y)}
  \end{align*}
  from which we can see that the conditional distribution of $Z$ given $Y=y$ has density as proposed.
  \end{proof}
  
\begin{figure}[h]
  \includegraphics[width=\textwidth]{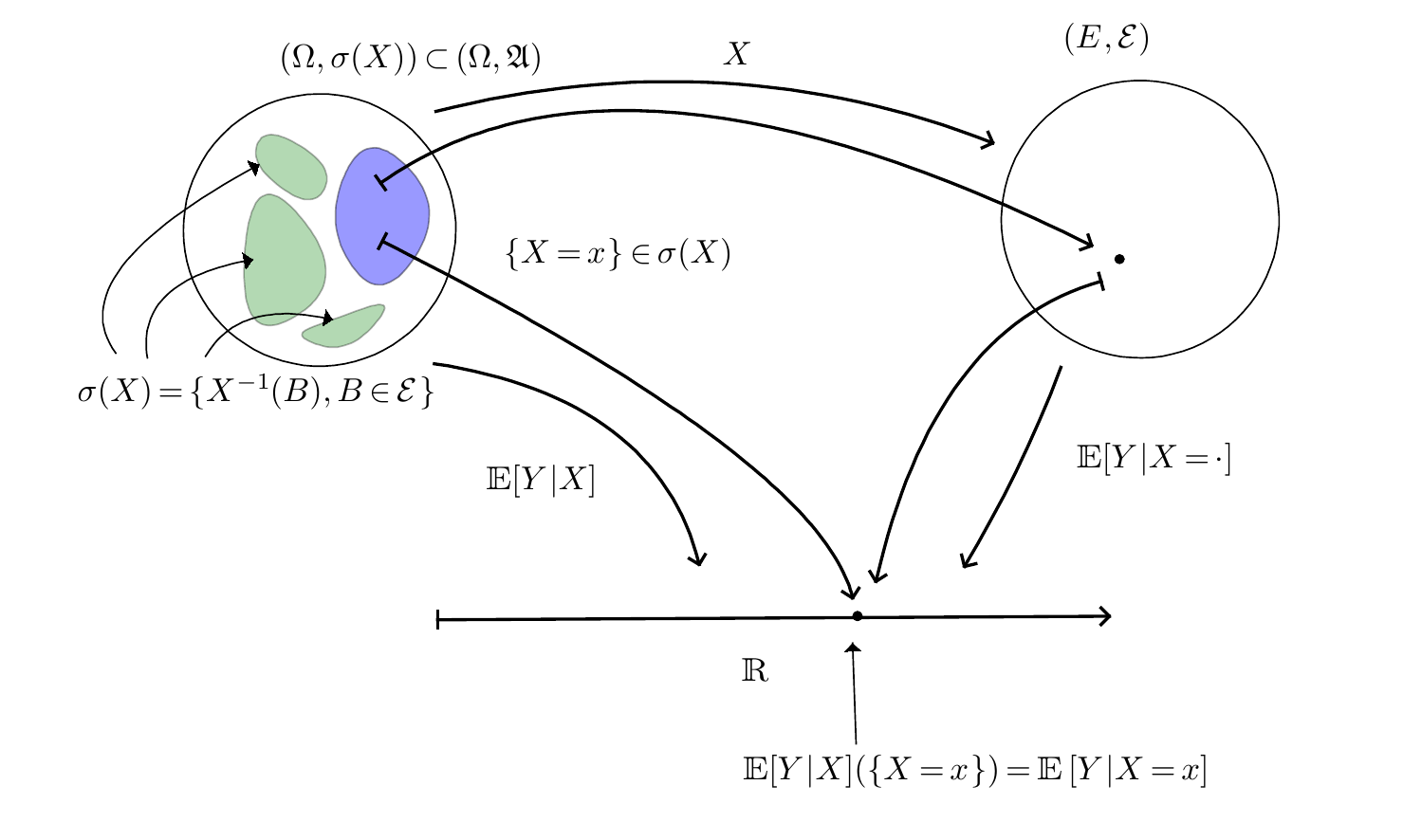}
  \caption{Conditional expectation w.r.t. singular events as concatenation of
  cond. exp. w.r.t. a random variable's $\sigma$-Algebra $\sigma (X)$ and the
  inverse image of $X$.}
\end{figure}

To recap, in order to calculate $\E[X|Y=y]$, we did the following:
\begin{enumerate}
\item Define $\E[X|\sigma(Y)]$ as an object in $L^1(\Omega, \sigma(Y), \P)$ which is constant on level sets of $Y$, i.e. on $\{\omega\in \Omega: Y(\omega) = y\}$ for any $y\in \R$. 
\item Define $y\mapsto \E[X|Y=y]$ as the value on such level sets. This is still not a pointwise definition because $\{Y=y\}$ is a set of measure $0$. In other words, the mapping $\varphi_X : y\mapsto \E[X|Y=y]$ is considered as an object $\varphi_X \in L^1(\R, \B(\R), \P\circ Y^{-1})$.
\item Show that (every version of) $y\mapsto \E[X|Y=y]$ is almost everywhere pointwise identical to something we can explicitly compute, i.e. the limits of $\E[X|Y\in (y-\epsilon, y+\epsilon)]$ for $\epsilon \to 0$.
\item Redefine $\varphi_X$ as this specific version which now is a function which can be evaluated pointwise.
\end{enumerate}

\subsection{A negative result for conditional expectations with respect to generic $\sigma$-algebras}

The procedure outlined above is strongly dependent on the fact that we conditioned on the sigma-algebra of a random variable. We could imagine doing something similar in a the more general case of conditioning with respect to an arbitrary sigma-algebra: Consider a random variable $X$ and an event $A$ with measure $0$. Analogously to above, we could choose a coarser sigma-algebra $\F$ such that $A\in \F$ and consider $\E[X|\F]$. Then we could try to insert $\omega\in A$ in this new random variable (similarly how we could insert an arbitrary $\omega\in\{Y=y\}$ into $\E[X|Y]$ in order to obtain $\E[X|Y=y]$. This can fail to work due to two different reasons: Either $\F$ is too fine or too coarse. $\sigma$-algebras generated by a random variable seem to be the only viable case.

We will use an allegory from real analysis to illustrate what can go wrong: Consider the function $f(x, y) = \theta$ where $x=r\cos\theta$ and $y=r\sin\theta$ for $r>0$ and $\theta\in[0,2\pi)$ (see figure \ref{fig:spiral}) defined on $\R^2\setminus\{0\}$. We will compare evaluation/extension of $f$ in $0$ with evaluation of a conditional expectation in a singular event. Note how both are well-defined as Lebesgue-integrable objects but they don't completely allow ``pointwise'' evaluation: $\E[X|\mathcal A]$ is not meaningful on singular sets and $f$ is not defined in $0$. We will see how choosing a $\sigma$-algebra $\mathcal A$ is similar to restricting the space $\R^2$ for the function example in order to allow extension and evaluation of $f$ in $0$.
\begin{figure}[hbtp]
\centering
\includegraphics[width=0.5\textwidth]{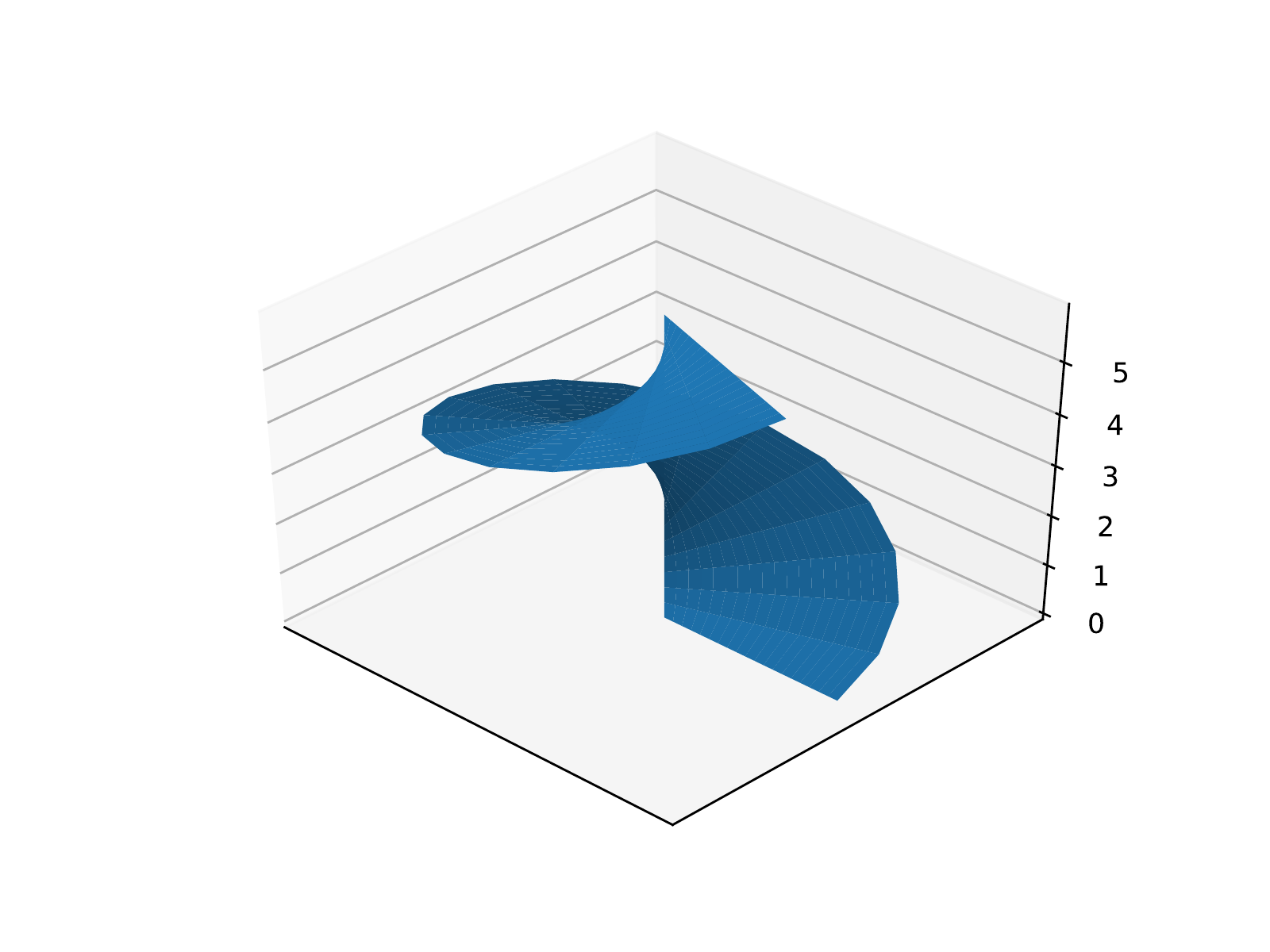}
\caption{The function $f(r, \theta) = \theta$ plotted in cartesian coordinates.}
\label{fig:spiral}
\end{figure}
\paragraph{$\F$ is too fine.}
We can, for example, choose as conditioning algebra the original $\sigma$-algebra $\F = \A$. Then $\E[X|\F] = X$. Now $\E[X|\F](\omega) = X(\omega)$ depends strongly on the specific choice of $\omega \in A$. Indeed, if it were arbitrary which $\omega\in A$ we pick, this would mean that $A\in \sigma(X)$, i.e. $A = \{X\in B\}$ for some $B\in \B(\R)$. But then we would actually be conditioning $X$ on $X$ itself, which is not interesting. Hence, the choice of $\omega$ is not well-defined.

Secondly, we cannot do any kind of continuation argument as in the case above: There may be multiple ways of approximating $A$ by sets in $\F$ with nonzero measure. The Borel--Kolmogorov paradox \cite{billingsley2008probability,gyenis2017conditioning,proschan1998expect,rao2005conditional} is witness to that fact. 

This can be likened to the continuity properties of the function $f$ from above. The value of $f(0,0)$ does not only not exist, but for any values $z\in [0,2\pi]$ we can find a sequence $z_n$ such that $\lim f(z_n) = z$. Similarly: If the sigma-algebra on which we condition is too fine, then there are too many possibilities of approximating the event which we condition on. If we condition on some random variable, on the other hand, then there really is just one ``direction'' of approximation. This is similar to the case where we only look at the function $f$ on one line through the origin: $f|_L$ where $L = \{(x,y): y = l\cdot x\}\setminus\{0\}$ for some $l$. Then the function $f|_{L}$ is continuous and can be extended in the origin.

\paragraph{$\F$ is too coarse.}
What if we choose $\F = \{\emptyset, A, A^c, \Omega\}$? This is a valid sigma-algebra and it seems that we should be able to define $\E[X|A] = \E[X|\F](\omega)$ for any $\omega \in A$. As the conditional expectation needs to be $\F$-measurable, this means that it only takes two values, depending on whether $\omega \in A$ or $\omega \in A^c$. But as it is also only unique up to sets of measure $0$, the value on $A$ is completely arbitrary: We can without hesitation just set \[\E[X|\F](\omega) = \begin{cases}17 & \text{ if } \omega \in A\\ \E[X] &\text{ else }\end{cases}\]
or any other numerical value on $A$. This is not new: Before, too, we could not assign a fixed numerical value to the set of measure $0$. But now we don't have any approximating sets: The only allowed sets are $A$ and $A^c$. In the situation of figure \ref{fig:spiral}, this would amount to restricting $f$ on the set $\{0\} \cup S^1$, where $S^1$ is the sphere (without interior) of radius $1$. Here, we cannot give the function $f$ a value in $0$ because there is no sequence of allowed points converging to $0$.

This can be compared to \cite{tjur1974conditional,tjur1975constructive} where conditional distributions were defined by approximation via nets of neighborhoods, but only in the more specific context of conditional distributions with respect to another random variable.

\paragraph{$\F$ is just right} This is the case of $\F = \sigma(Y)$ for some $Y$: This restricts the family of neighborhoods of $A$ to a manageable size such that uniqueness of the approximation holds but not too much that there is no such sequence of neighborhoods.

\bibliographystyle{abbrv}
\bibliography{lit}

\end{document}